\documentclass[12pt,leqno]{article}
\usepackage{hyperref}
\usepackage{soul}
\usepackage[doc]{optional}
\usepackage{xcolor}
\definecolor{labelkey}{rgb}{0,0.08,0.45}
\definecolor{rekey}{rgb}{0,0.6,0.0}
\definecolor{Brown}{rgb}{0.45,0.0,0.05}
\usepackage{exscale,relsize}
\usepackage{amsmath}
\usepackage{amsfonts}
\usepackage{amssymb}
\usepackage{calc}
\usepackage{theorem}
\usepackage{pifont}      
\usepackage{graphicx}
\DeclareMathOperator{\weakstarly}{\stackrel{\mathrm{w*}}{\rightharpoondown}}

\newcommand{\scal}[2]{\langle{{#1},{#2}}\rangle}

\newcommand{\RR}{\ensuremath{\mathbb R}}

\newcommand{\RX}{\ensuremath{\,\left]-\infty,+\infty\right]}}
\newcommand{\RXX}{\ensuremath{\,\left[-\infty,+\infty\right]}}

\newcommand{\NN}{\ensuremath{\mathbb N}}

\newcommand{\menge}[2]{\big\{{#1} \mid {#2}\big\}}

\newcommand{\To}{\ensuremath{\rightrightarrows}}

\newcommand{\aff}{\operatorname{aff}}

\newcommand{\dom}{\ensuremath{\operatorname{dom}}}

\newcommand{\gra}{\ensuremath{\operatorname{gra}}}

\newcommand{\intdom}{\ensuremath{\operatorname{int}\operatorname{dom}}\,}
\newcommand{\inte}{\ensuremath{\operatorname{int}}}

\newcommand{\bd}{\ensuremath{\operatorname{bdry}}}

\renewcommand{\phi}{\ensuremath{\varphi}}

\newtheorem{theorem}{Theorem}[section]

\newtheorem{fact}[theorem]{Fact}
\newtheorem{corollary}[theorem]{Corollary}

\theoremstyle{plain}{\theorembodyfont{\rmfamily}
}
\theoremstyle{plain}{\theorembodyfont{\rmfamily}
}
\theoremstyle{plain}{\theorembodyfont{\rmfamily}
}
\theoremstyle{plain}{\theorembodyfont{\rmfamily}
\newtheorem{example}[theorem]{Example}}
\theoremstyle{plain}{\theorembodyfont{\rmfamily}
\newtheorem{remark}[theorem]{Remark}}
\newtheorem{problem}[theorem]{Open problem}
\theoremstyle{plain}{\theorembodyfont{\rmfamily}
}

\begin{document}


\title{\sffamily{
The sum of a maximally monotone
 linear relation and\\ the subdifferential of a proper lower semicontinuous\\ convex function is maximally monotone}}

\author{Liangjin\
Yao\thanks{Mathematics, Irving K.\ Barber School, UBC Okanagan,
Kelowna, British Columbia V1V 1V7, Canada.
E-mail:  \texttt{ljinyao@interchange.ubc.ca}.}}
 \vskip 3mm

\date{October 19, 2010}
\maketitle

\begin{abstract} \noindent
The most important open problem in Monotone Operator Theory
concerns the maximal monotonicity of the sum of two
maximally monotone operators provided that
Rockafellar's constraint qualification holds.

In this paper, we prove the maximal monotonicity of $A+\partial f$
provided that
$A$ is a maximally monotone linear relation, and $f$ is a proper lower semicontinuous convex function
 satisfying
$\dom A\cap\inte\dom \partial f\neq\varnothing$. Moreover, $A+\partial f$ is of type (FPV).
The maximal monotonicity of $A+\partial f$ when $\intdom A\cap\dom \partial f\neq\varnothing$
follows from a result by
Verona and Verona, which the present work complements.

\end{abstract}



\section{Introduction}

Throughout this paper, we assume that
$X$ is a real Banach space with norm $\|\cdot\|$,
that $X^*$ is the continuous dual of $X$, and
that $X$ and $X^*$ are paired by $\scal{\cdot}{\cdot}$.
Let $A\colon X\To X^*$
be a \emph{set-valued operator} (also known as multifunction)
from $X$ to $X^*$, i.e., for every $x\in X$, $Ax\subseteq X^*$,
and let
$\gra A = \menge{(x,x^*)\in X\times X^*}{x^*\in Ax}$ be
the \emph{graph} of $A$.
Recall that $A$ is  \emph{monotone} if
\begin{equation}
\scal{x-y}{x^*-y^*}\geq 0,\quad \forall (x,x^*)\in \gra A\;\forall (y,y^*)\in\gra A,
\end{equation}
and \emph{maximally monotone} if $A$ is monotone and $A$ has no proper monotone extension
(in the sense of graph inclusion).
Let $A:X\rightrightarrows X^*$ be monotone and $(x,x^*)\in X\times X^*$.
 We say $(x,x^*)$ is \emph{monotonically related to}
$\gra A$ if
\begin{align*}
\langle x-y,x^*-y^*\rangle\geq0,\quad \forall (y,y^*)\in\gra A.\end{align*}
Let $A:X\rightrightarrows X^*$ be maximally monotone. We say $A$ is
\emph{of type (FPV)} if  for every open convex set $U\subseteq X$ such that
$U\cap \dom A\neq\varnothing$, the implication
\begin{equation*}
x\in U\,\text{and}\,(x,x^*)\,\text{is monotonically related to $\gra A\cap (U\times X^*)$}
\Rightarrow (x,x^*)\in\gra A
\end{equation*}
holds.
We say $A$ is a \emph{linear relation} if $\gra A$ is a linear subspace.
Monotone operators have proven to be a key class of objects
in modern Optimization and Analysis; see, e.g.,
the books
\cite{BorVan,BurIus,ButIus,ph,Si,Si2,RockWets,Zalinescu,Zeidler}
and the references therein.
We adopt standard notation used in these books:
$\dom A= \menge{x\in X}{Ax\neq\varnothing}$ is the \emph{domain} of $A$.
Given a subset $C$ of $X$,
$\inte C$ is the \emph{interior} of $C$,
$\bd{C}$ is the \emph{boundary}, $\aff{C}$ is the \emph{affine hull}, and
$\overline{C}$ is the norm \emph{closure} of $C$.
We set $C^{\bot}=
\{x^*\in X^* \mid(\forall c\in C)\, \langle x^*, c\rangle=0\}$
and $S^{\bot}=
\{x^{**}\in X^{**} \mid(\forall s\in S)\, \langle x^{**}, s\rangle=0\}$ for a set  $S\subseteq X^*$.
The \emph{indicator function} of $C$, written as $\iota_C$, is defined
at $x\in X$ by
\begin{align}
\iota_C (x)=\begin{cases}0,\,&\text{if $x\in C$;}\\
\infty,\,&\text{otherwise}.\end{cases}\end{align}
If $D\subseteq X$, we set $C-D=\{x-y\mid x\in C, y\in D\}$.
  For every $x\in X$, the normal cone operator of $C$ at $x$
is defined by $N_C(x)= \menge{x^*\in
X^*}{\sup_{c\in C}\scal{c-x}{x^*}\leq 0}$, if $x\in C$; and $N_C(x)=\varnothing$,
if $x\notin C$.
For $x,y\in X$, we set $\left[x,y\right]=\{tx+(1-t)y\mid 0\leq t\leq 1\}$.
 Given $f\colon X\to \RX$, we set
$\dom f= f^{-1}(\RR)$ and
$f^*\colon X^*\to\RXX\colon x^*\mapsto
\sup_{x\in X}(\scal{x}{x^*}-f(x))$ is
the \emph{Fenchel conjugate} of $f$.
If $f$ is convex and $\dom f\neq\varnothing$, then
   $\partial f\colon X\To X^*\colon
   x\mapsto \menge{x^*\in X^*}{(\forall y\in
X)\; \scal{y-x}{x^*} + f(x)\leq f(y)}$
is the \emph{subdifferential operator} of $f$.
We also set $P_X: X\times X^*\rightarrow X\colon (x,x^*)\mapsto x$.
Finally,  the \emph{open unit ball} in $X$ is denoted by
$U_X= \menge{x\in X}{\|x\|< 1}$, and $\NN=\{1,2,3,\ldots\}$.

Let $A$ and $B$ be maximally monotone operators from $X$ to
$X^*$.
Clearly, the \emph{sum operator} $A+B\colon X\To X^*\colon x\mapsto
Ax+Bx = \menge{a^*+b^*}{a^*\in Ax\;\text{and}\;b^*\in Bx}$
is monotone.
Rockafellar's \cite[Theorem~1]{Rock70} guarantees maximal monotonicity
of $A+B$ under
\emph{Rockafellar's constraint qualification}
$\dom A \cap\intdom B\neq \varnothing$ when $X$ is reflexive
--- this result is often referred to as ``the sum theorem''.
The most famous open problem concerns the maximal monotonicity of $A+B$ in
nonreflexive Banach spaces when Rockafellar's constraint qualification
holds.
See Simons' monograph
\cite{Si2} and \cite{Bor1, Bor2, Bor3, ZalVoi, Yao3}
for a comprehensive account of some recent developments.

Now we focus on the  case when
$A$ is a maximally monotone linear relation,
and $f$ is a proper lower semicontinuous convex function
 such that
$\dom A\cap\inte\dom \partial f\neq\varnothing$. We show that $A+\partial f$ is maximally monotone.
Linear relations have recently become a center of attention in Monotone Operator Theory; see, e.g.,
\cite{BB,BBW,BWY2,BWY3,BWY4,BWY9, BWY7, BWY8, PheSim, Si3,Svaiter,Voisei06b,Voisei06,VZ,Yao}
and Cross' book \cite{Cross} for general background on linear
relations.

The remainder of this paper is organized as follows.
In Section~\ref{s:aux}, we collect auxiliary results for future reference
and for the
reader's convenience.
The main result (Theorem~\ref{PGV:1}) is proved
in Section~\ref{s:main}.

\section{Auxiliary Results}
\label{s:aux}

\begin{fact}[Rockafellar] \label{f:F4}
\emph{(See {\cite[Theorem~3(b)]{Rock66}},
{\cite[Theorem~18.1]{Si2}}, or
{\cite[Theorem~2.8.7(iii)]{Zalinescu}}.)}
Let $f,g: X\rightarrow\RX$ be proper convex functions.
Assume that there exists a point $x_0\in\dom f \cap \dom g$
such that $g$ is continuous at $x_0$.
Then  $\partial (f+g)=\partial f+\partial g$.
\end{fact}

\begin{fact}[Rockafellar]\label{SubMR}\emph{(See \cite[Theorem~A]{Rock702},
 \cite[Theorem~3.2.8]{Zalinescu}, \cite[Theorem~18.7]{Si2} or \cite[Theorem~2.1]{MSV})}
 Let $f:X\rightarrow\RX$ be a proper lower semicontinuous convex function.
Then $\partial f$ is maximally monotone.
\end{fact}

\begin{fact}\emph{(See \cite[Theorem~2.28]{ph}.)}
\label{pheps:11}Let $A:X\To X^*$ be  monotone with $\inte\dom A\neq\varnothing$.
Then $A$ is locally bounded at $x\in\inte\dom A$, i.e., there exist $\delta>0$ and $K>0$ such that
\begin{align*}\sup_{y^*\in Ay}\|y^*\|\leq K,\quad \forall y\in (x+\delta U_X)\cap \dom A.
\end{align*}
\end{fact}

\begin{fact}[Fitzpatrick]
\emph{(See {\cite[Corollary~3.9]{Fitz88}}.)}
\label{f:Fitz}
Let $A\colon X\To X^*$ be maximally monotone,  and set
\begin{equation}
F_A\colon X\times X^*\to\RX\colon
(x,x^*)\mapsto \sup_{(a,a^*)\in\gra A}
\big(\scal{x}{a^*}+\scal{a}{x^*}-\scal{a}{a^*}\big),
\end{equation}
 the \emph{Fitzpatrick function} associated with $A$.
Then for every $(x,x^*)\in X\times X^*$, the inequality
$\scal{x}{x^*}\leq F_A(x,x^*)$ is true,
and the equality holds if and only if $(x,x^*)\in\gra A$.
\end{fact}

\begin{fact}
\emph{(See \cite[Theorem~3.4 and Corollary~5.6]{Voi1}, or \cite[Theorem~24.1(b)]{Si2}.)}
\label{f:referee1}
Let $A, B:X\To X^*$ be maximally monotone operators. Assume
$\bigcup_{\lambda>0} \lambda\left[P_X(\dom F_A)-P_X(\dom F_B)\right]$
is a closed subspace.
If
\begin{equation}
F_{A+B}\geq\langle \cdot,\,\cdot\rangle\;\text{on \; $X\times X^*$},
\end{equation}
then $A+B$ is maximally monotone.
\end{fact}

\begin{fact}[Simons]
\emph{(See \cite[Theorem~48.6(a)]{Si2}.)}
\label{slopeD}
Let $f:X\to\RX$ be proper, lower semicontinuous, and convex.
Let $(x,x^*)\in X\times X^*$ and $\alpha>0$ be such that
$(x,x^*)\notin\gra \partial f$. Then for every $\varepsilon>0$,
there exists $(y,y^*)\in\gra \partial f$ with $y\neq x$ and
$y^*\neq x^*$ such that
\begin{align}
\left|\frac{\|x-y\|}{\|x^*-y^*\|}- \alpha\right|<\varepsilon\end{align}
 and
\begin{align}
\left|\frac{\langle x-y,x^*-y^*\rangle}{\|x-y\|\cdot\|x^*-y^*\|}+1\right|<\varepsilon.
\end{align}
\end{fact}

\begin{fact}[Simons]
\emph{(See \cite[Theorem~46.1]{Si2}.)}
\label{f:referee01}
Let $A:X\To X^*$ be a maximally monotone linear relation.
Then $A$ is of type (FPV).
\end{fact}

\begin{fact}\emph{(See \cite[Proposition~3.3 and Proposition~1.11]{ph}.)}
\label{pheps:1}Let $f:X\rightarrow\RX$ be a lower semicontinuous convex
 and $\intdom f\neq\varnothing$.
Then $f$ is continuous on $\intdom f$ and $\partial f(x)\neq\varnothing$ for every $x\in\intdom f$.
\end{fact}

\begin{fact}\emph{(See \cite[Lemma~2.9]{BWY9}.)}\label{rcf:001}
Let $A:X\To X^*$ be a maximally monotone linear relation,
and let $z\in X\cap (A0)^\bot$.
Then $z\in\overline{\dom A}$.
\end{fact}

\begin{fact}\emph{(See \cite[Lemma~2.10]{BWY9}.)}\label{rcf:01}
Let $A:X\To X^*$ be a  monotone linear relation,
and let $f:X\to\RX$ be a proper lower semicontinuous convex function.
Suppose that $\dom A \cap \inte \dom \partial f\neq \varnothing$,
$(z,z^*)\in X\times X^*$ is monotonically related to
$\gra (A+\partial f)$, and that $z\in\dom A$.
Then  $z\in  \dom \partial f$.
\end{fact}

\begin{fact}[Simons and Verona-Verona]
\emph{(See \cite[Thereom~44.1]{Si2} or \cite{VV1}.)}
\label{f:referee02a}
Let $A:X\To X^*$ be a maximally monotone. Suppose that
for every closed convex subset $C$ of $X$
with $\dom A \cap \inte C\neq \varnothing$, the operator
$A+N_C$ is maximally monotone.
Then $A$ is of type  (FPV).
\end{fact}

\begin{fact}\label{Ll:l1}\emph{(See \cite[lemma~2.5]{BWY4}.)}
Let $C$ be  a nonempty closed convex
subset of $X$ such that $\inte C\neq \varnothing$.
Let $c_0\in \inte C$ and suppose that $z\in X\smallsetminus C$.
Then there exists
$\lambda\in\left]0,1\right[$ such
that $\lambda c_0+(1-\lambda)z\in\bd C$.
\end{fact}
\section{Main Result}
\label{s:main}

\begin{theorem}\label{PGV:1}
Let $A:X\To X^*$ be a maximally monotone linear relation,
and let $f:X\To \RX$ be a proper lower semicontinuous convex function
with $\dom A\cap\inte\dom \partial f\neq\varnothing$.  Then
$A+\partial f$ is maximally monotone.
\end{theorem}
\begin{proof} After translating the graphs if necessary, we can and do assume that
$0\in\dom A\cap\inte\dom \partial f$ and that $(0,0)\in\gra A\cap\gra \partial f$.
By Fact~\ref{f:Fitz} and Fact~\ref{SubMR}, $\dom A\subseteq P_X(\dom F_A)$ and
 $\dom \partial f\subseteq P_X(\dom F_{\partial f})$.
Hence,
\begin{align}\bigcup_{\lambda>0} \lambda
\big(P_X(\dom F_A)-P_X(\dom F_{\partial f})\big)=X.\end{align}
Thus, by Fact~\ref{SubMR} and Fact~\ref{f:referee1}, it suffices to show that
\begin{equation} \label{e0:ourgoal}
F_{A+ \partial f}(z,z^*)\geq \langle z,z^*\rangle,\quad \forall(z,z^*)\in X\times X^*.
\end{equation}
Take $(z,z^*)\in X\times X^*$.
Then
\begin{align}
&F_{A+\partial f}(z,z^*)\nonumber\\
&=\sup_{\{x,x^*,y^*\}}\left[\langle x,z^*\rangle+\langle z,x^*\rangle-\langle x,x^*\rangle
+\langle z-x, y^*\rangle -\iota_{\gra A}(x,x^*)-\iota_{\gra
\partial f}(x,y^*)\right].\label{see:1}
\end{align}
Assume to the contrary  that
\begin{align}
F_{A+\partial f}(z,z^*)+\lambda<\langle z,z^*\rangle,\label{See:1a4}
\end{align}
where $\lambda>0$.

Thus by \eqref{See:1a4},
\begin{align}
(z,z^*)\,\text{ is monotonically related to $\gra (A+\partial f)$}.\label{SDF:47}\end{align}

We claim that \begin{align}z\notin \dom A.\label{LSD:1}\end{align}
 Indeed, if $z\in\dom A$, apply \eqref{SDF:47} and Fact~\ref{rcf:01}
to get $z\in \dom\partial f$. Thus $z\in\dom A\cap\dom\partial f$ and hence
$F_{A+\partial f}(z,z^*)\geq\langle z,z^*\rangle$ which contradicts
\eqref{See:1a4}. This verifies \eqref{LSD:1}.

By \eqref{See:1a4} and the assumption that $(0,0)\in\gra A\cap\gra \partial f$,
we have
\begin{align*}
\sup\left[\langle 0,z ^*\rangle+\langle z,A0\rangle
-\langle 0, A0\rangle+\langle z, \partial f(0)\rangle\right]
= \sup_{a^*\in A0,b^*\in\partial f(0)}
\left[\langle z,a^*\rangle+\langle z, b^*\rangle\right]
<\langle z,z^*\rangle.
\end{align*}
Thus, because $A0$ is a linear subspace,
\begin{align}z\in X\cap (A0)^\bot.\label{ree:2}\end{align}
Then, by
Fact~\ref{rcf:001}, we have
\begin{align}z\in\overline{\dom A}.\label{ree:5}\end{align}
 Combine \eqref{LSD:1} and \eqref{ree:5},
 \begin{align}
 z\in\overline{\dom A}\backslash{\dom A}.\label{LSD:2}\end{align}
Set
 \begin{align}
U_n=  z+\tfrac{1}{n}U_X,\quad \forall n\in\NN\label{FD:1}.
\end{align}
By \eqref{LSD:2}, $(z,z^*)\notin\gra A$ and $U_n\cap\dom A\neq\varnothing$.
  Since $z\in U_n$ and $A$ is type of (FPV)
 by Fact~\ref{f:referee01},
there exist $(a_n, a^*_n)\in\gra A$  with $a_n\in U_n,  n\in\NN$ such that
\begin{align}
\langle z,a^*_n\rangle+\langle a_n,z^*\rangle-\langle a_n, a^*_n\rangle
>\langle z,z^*\rangle.\label{LSD:3}\end{align}
Then we have
 \begin{align}
 a_n\rightarrow z.\label{LSD:8}\end{align}

Now we claim that \begin{align}z\in\overline{\dom \partial f}.\end{align}

Suppose to the contrary that $z\not\in\overline{\dom\partial f}$.
By the Br{\o}ndsted-Rockafellar Theorem (see \cite[Theorem~3.17]{ph} or \cite[Theorem~3.1.2]{Zalinescu}),
 $\overline{\dom\partial f}=\overline{\dom f}$.
By $0\in\inte \dom \partial f\subseteq\inte\dom f\subseteq\inte\overline{\dom f}$, then by Fact~\ref{Ll:l1},
 there exists $\delta\in\left]0,1\right[$
such that
\begin{align}\delta z \in\bd\overline{\dom f}.\label{LSD:4}\end{align}

Set $g_n:X\rightarrow \RX$ by
\begin{align}
g_n=f+\iota_{\left[0,a_n\right]},\quad n\in\NN\end{align}

Since $z\notin\overline{\dom f}$, $z\not\in\dom f\cap \left[0,a_n\right]=\dom g_n$.
Thus $(z,z^*)\notin \gra\partial g_n$.
Then by Fact~\ref{slopeD},
  there exist $\beta_n\in\left[0,1\right]$ and $x^*_n
\in\partial g_n (\beta_n a_n)$  with $x^*_n\neq z^*$ and $\beta_n a_n\neq z$ such that
\begin{align}
\frac{\|z-\beta_n a_n\|}{\|z^*-x^*_n\|}&\geq n\label{LSD:5}\\
\frac{\langle z-\beta_n a_n,z^*-x^*_n\rangle}{\|z-\beta_n a_n\|\cdot\|z^*-x^*_n\|}&<-\tfrac{3}{4}.
\label{LSD:7}
\end{align}

By \eqref{LSD:8}, $\|z-\beta_n a_n\|$ is bounded. Then by \eqref{LSD:5}, we have
\begin{align}
x^*_n\rightarrow z^*. \label{LSD:9}\end{align}
Since $0\in\inte\dom f$,  $f$ is continuous at $0$
 by Fact~\ref{pheps:1}. Then by $0\in\dom f\cap \dom \iota_{\left[0,a_n\right]}$
  and Fact~\ref{f:F4}, we have that there exist
$w^*_n\in\partial f(\beta_n a_n)$ and $v^*_n\in
\partial \iota_{\left[0,a_n\right]} (\beta_n a_n)$ such that
$x^*_n=w^*_n+v^*_n$.
Then by \eqref{LSD:9},
\begin{align}
w^*_n+v^*_n\rightarrow z^*.\label{LSD:15}\end{align}

Since $\beta_n\in \left[0,1\right]$,  there exists  a
convergent subsequence of $(\beta_n)_{n\in\NN}$, which, for convenience, we
still denote by $(\beta_n)_{n\in\NN}$. Then
$\beta_n\rightarrow \beta$, where $\beta\in\left[0,1\right]$.
Then by \eqref{LSD:8},
\begin{align}
\beta_n a_n\rightarrow \beta z.\label{LSD:11}\end{align}

We claim that
\begin{align}
\beta\leq\delta<1.\label{LSD:12}\end{align}

In fact, suppose to the contrary that $\beta>\delta$. By \eqref{LSD:11}, $\beta z\in\overline{\dom f}$.
 Then by $0\in\inte\dom f$ and \cite[Theorem~1.1.2(ii)]{Zalinescu},
$\delta z=\tfrac{\delta}{\beta } \beta z\in \inte\overline{\dom f},$
which contradicts \eqref{LSD:4}.

We can and do suppose that $\beta_n< 1$  for every $n\in\NN$.
Then by $v^*_n\in \partial \iota_{\left[0,a_n\right]}(\beta_n a_n)$, we have
\begin{align}
\langle v^*_n, a_n-\beta_n a_n\rangle\leq 0.\end{align}
Dividing by $(1-\beta_n)$ on both sides of the above inequality, we have
\begin{align}
\langle v^*_n,  a_n\rangle\leq 0.\label{LSD:14}\end{align}

Since $(0,0)\in\gra A$, $\langle a_n, a^*_n\rangle\geq 0, \forall n\in\NN$.
Then by \eqref{LSD:3}, we have
\begin{align}&\langle  z, \beta_n a^*_n\rangle+
\langle \beta_n a_n,z^*\rangle-\beta^2_n \langle a_n, a^*_n\rangle
 \geq\langle \beta_n z,a^*_n\rangle+\langle \beta_n a_n,z^*\rangle
 -\beta_n \langle a_n, a^*_n\rangle
\geq\beta_n \langle z,z^*\rangle.\label{LSD:17}\end{align}
Then by \eqref{LSD:17},
\begin{align}
\langle  z-\beta_n a_n, \beta_n a^*_n\rangle\geq\langle \beta_n z-\beta_n a_n, z^*\rangle.
\label{LSD:20}
\end{align}
Since $\gra A$ is a linear subspace and $(a_n,a^*_n)\in\gra A$, $(\beta_n a_n, \beta_n a^*_n)\in\gra A$.
 By \eqref{See:1a4}, we have
 \begin{align*}
 \lambda<&\langle z-\beta_n a_n, z^*-w^*_n-\beta_n a^*_n\rangle\\
 &= \langle z-\beta_n a_n, z^*-w^*_n\rangle+\langle z-\beta_n a_n, -\beta_n a^*_n\rangle\\
 &< -\tfrac{3}{4}\|z-\beta_n a_n\|\cdot\|z^*-w^*_n-v^*_n\|+\langle z-\beta_n a_n, v^*_n\rangle
 +\langle z-\beta_n a_n, -\beta_n a^*_n\rangle\quad\text{(by \eqref{LSD:7})}\\
 &\leq -\tfrac{3}{4}\|z-\beta_n a_n\|\cdot\|z^*-w^*_n-v^*_n\|+\langle z-\beta_n a_n, v^*_n\rangle
 -\langle \beta_n z-\beta_n a_n, z^*\rangle\quad\text{(by \eqref{LSD:20})}.
 \end{align*}
Then
\begin{align}
 \lambda
 <
 \langle z-\beta_n a_n, v^*_n\rangle-\langle \beta_n z-\beta_n a_n, z^*\rangle.\label{LSD:21}
 \end{align}

Now we consider two cases:

\emph{Case 1}:
 $(w^*_n)_{n\in\NN}$ is bounded.

 By \eqref{LSD:15}, $(v^*_n)_{n\in\NN}$ is bounded. By the Banach-Alaoglu Theorem
(see \cite[Theorem~3.15]{Rudin}), there  exist a weak* convergent \emph{subnet}
$(v^*_\gamma)_{\gamma\in\Gamma}$ of $(v^*_n)_{n\in\NN}$, say
\begin{align}{v^*_\gamma}\weakstarly  v^*_{\infty}\in X^*.\label{FCGG:9}\end{align}

Combine \eqref{LSD:8}, \eqref{LSD:11} and \eqref{FCGG:9}, and
pass the limit along the subnet of \eqref{LSD:21} to get that

\begin{align}
 \lambda\leq
 \langle z-\beta z, v^*_{\infty}\rangle.\label{LSD:22}
 \end{align}
By \eqref{LSD:12},  divide by $(1-\beta)$ on both sides of \eqref{LSD:22} to get
\begin{align}
\langle z,v^*_{\infty}\rangle\geq\tfrac{\lambda}{1-\beta}>0.\label{LSD:23}
\end{align}
On the other hand, by \eqref{LSD:8} and \eqref{FCGG:9}, passing the limit
along the subnet of \eqref{LSD:14} to get that
\begin{align}
\langle v^*_{\infty}, z\rangle\leq 0,\end{align}
which contradicts \eqref{LSD:23}.

\emph{Case 2}:
 $(w^*_n)_{n\in\NN}$ is unbounded.

Since $(w^*_n)_{n\in\NN}$ is unbounded and
after passing to a subsequence if necessary, we assume that
$\|w^*_n\|\neq 0,\forall n\in\NN$ and that $\|w^*_n\|\rightarrow +\infty$.

By the Banach-Alaoglu Theorem
again, there  exist a weak* convergent \emph{subnet}
$(w^*_\nu)_{\nu\in I }$ of $(w^*_n)_{n\in\NN}$, say
\begin{align}\frac{w^*_{\nu}}{\|w^*_{\nu}\|}\weakstarly  w^*_{\infty}\in X^*.\label{FCGG:LSD9}\end{align}

By $0\in\inte\dom \partial f $ and Fact~\ref{pheps:11},
there exist $\rho>0$ and $M>0$ such that
\begin{align}
\partial f(y)\neq\varnothing \quad\text{and}\quad
\sup_{y^*\in \partial f(y)}\|y^*\|\leq M,\quad \forall y\in \rho
U_X.\label{RVT:10a}
\end{align}
Then by $w^*_n\in\partial f(\beta_n a_n)$, we have
\begin{align}
&\langle \beta_n a_n-y, w^*_n-y^*\rangle\geq0,\quad \forall y\in
\rho U_X, y^*\in \partial f(y), n\in\NN\nonumber\\
&\Rightarrow  \langle \beta_n a_n, w^*_n\rangle-\langle y, w^*_n\rangle
+\langle \beta_n a_n-y, -y^*\rangle\geq0,\quad\forall y\in
\rho U_X, y^*\in \partial f(y),  n\in\NN\nonumber\\
&\Rightarrow   \langle \beta_n a_n, w^*_n\rangle-
\langle y, w^*_n\rangle\geq\langle \beta_n a_n-y, y^*\rangle,\quad\forall y\in
\rho U_X, y^*\in \partial f(y),  n\in\NN\nonumber\\
&\Rightarrow \langle \beta_n a_n, w^*_n\rangle-
\langle y, w^*_n\rangle\geq -(\|\beta_n a_n\|+\rho) M, \quad\forall y\in \rho
U_X,  n\in\NN\quad\text{(by \eqref{RVT:10a})}\nonumber\\
&\Rightarrow \langle \beta_n a_n, w^*_n\rangle
\geq \langle y, w^*_n\rangle -(\|\beta_n a_n\|+\rho) M, \quad\forall y\in \rho
U_X,  n\in\NN\nonumber\\
&\Rightarrow \langle \beta_n a_n, w^*_n\rangle
\geq \rho\|w^*_n\|-(\|\beta_n a_n\|+\rho) M, \quad \forall n\in\NN\nonumber\\
&\Rightarrow \langle \beta_n a_n, \tfrac{w^*_n}{\|w^*_n\|}\rangle
\geq\rho-\frac{(\|\beta_n a_n\|+\rho) M}{\|w^*_n\|}, \quad \forall n\in\NN.\label{FCG:3}
\end{align}

Combining \eqref{LSD:11} and \eqref{FCGG:LSD9}, taking the limit in \eqref{FCG:3} along the subnet, we obtain
\begin{align}
\langle \beta z, w^*_{\infty}\rangle
\geq\rho.\label{LSD:30}\end{align}
Then we have $\beta\neq0$ and thus $\beta>0$. Then by \eqref{LSD:30},
\begin{align}
\langle  z, w^*_0\rangle
\geq\tfrac{\rho}{\beta}>0.\label{LSD:34}\end{align}

By \eqref{LSD:15} and $\tfrac{z^*}{\|w^*_n\|}\rightarrow 0$, we have
\begin{align}
\frac{w^*_n}{\|w^*_n\|}+\frac{v^*_n}{\|w^*_n\|}\rightarrow 0.
\label{LSD:31}
\end{align}
By\eqref{FCGG:LSD9}, taking the weak$^{*}$ limit in \eqref{LSD:31} along the subnet, we obtain
\begin{align}
\frac{v^*_\nu}{\|w^*_\nu\|}\weakstarly -w^*_{\infty}.
\label{LSD:32}
\end{align}

Dividing by $\|w^*_n\|$ on the both sides of \eqref{LSD:21}, we get that
\begin{align}
 \frac{\lambda}{\|w^*_n\|}
 <
 \langle z-\beta_n a_n, \frac{v^*_n}{\|w^*_n\|}\rangle-
 \frac{\langle \beta_n z-\beta_n a_n, z^*\rangle}{\|w^*_n\|}.\label{LSD:35}
 \end{align}
Combining \eqref{LSD:11}, \eqref{LSD:8} and \eqref{LSD:32},
taking the limit in \eqref{LSD:35} along the subnet, we obtain
\begin{align}
\langle z-\beta z, -w^*_{\infty}\rangle\geq0.\label{LSD:37}
 \end{align}
By \eqref{LSD:12} and \eqref{LSD:37},
\begin{align}
\langle z, -w^*_{\infty}\rangle\geq0,\label{LSD:38}\end{align}
which contradicts \eqref{LSD:34}.

Altogether $ z\in\overline{\dom\partial f}=\overline{\dom f}$.

Next, we show that
\begin{align}
F_{A+\partial f}(t z,tz^*)\geq t^2\langle z,z^*\rangle,
\quad\forall t\in\left]0,1\right[.\label{See:10}\end{align}
Let $t\in\left]0,1\right[$.
By  $0\in\inte\dom f$
and \cite[Theorem~1.1.2(ii)]{Zalinescu}, we have
\begin{align}
tz\in\inte\dom f\label{ReAu:1}.
\end{align}
By Fact~\ref{pheps:1},
\begin{align}tz\in\inte\dom \partial f.
\end{align}

Set
 \begin{align}
H_n= t z+\tfrac{1}{n}U_X,\quad \forall n\in\NN\label{FD:1}.
\end{align}
Since $\dom A$ is a linear subspace, $tz\in\overline{\dom A}\backslash {\dom A}$ by \eqref{LSD:2}.
Then $H_n\cap \dom A\neq\varnothing$.
Since $(tz, t z^*)\notin\gra A$ and $tz\in H_n$,
$A$ is of type (FPV) by Fact~\ref{f:referee01},
 there exists $(b_n,b^*_n)\in\gra A$
such that $b_n\in H_n$  and
\begin{align}
\langle t z,b^*_n\rangle+\langle b_n,t z^*\rangle-
\langle b_n,b^*_n\rangle>t^2\langle z,z^*\rangle,\quad \forall n\in\NN.\label{see:20}
\end{align}
Since $t z\in\inte\dom \partial f$ and $b_n\rightarrow t z$, by Fact~\ref{pheps:11},
 there exist $N\in\NN$ and  $K>0$ such that
\begin{align}b_n\in\inte\dom \partial f \quad\text{and}\quad \sup_{v^*\in \partial f(b_n)}\| v^*\|
\leq K,\quad \forall n\geq N.\label{See:1a2}
\end{align}
Hence
\begin{align}
&F_{A+\partial f}(t z,t z^*)\nonumber\\
&\geq\sup_{\{ c^*\in \partial f(b_n)\}}\left[\langle b_n,t z^*\rangle
+\langle t z,b_n^*\rangle-\langle b_n,b_n^*\rangle
+\langle t z-b_n, c^*\rangle \right],\quad \forall n\geq N\nonumber\\
&\geq\sup_{\{ c^*\in \partial f(b_n)\}}\left[t^2\langle z,z^*\rangle
+\langle t z-b_n,c^*\rangle \right],\quad \forall n\geq N
\quad\text{(by \eqref{see:20})}\nonumber\\
&\geq\sup\left[t^2\langle z,z^*\rangle
-K\|t z-b_n\| \right],\quad \forall n\geq N\quad\text{(by \eqref{See:1a2})}\nonumber\\
&\geq t^2\langle z,z^*\rangle\quad\text{(by $b_n\rightarrow t z$)}.\label{See:1a3}
\end{align}
Hence $
F_{A+\partial f}(t z,t z^*)\geq t^2\langle z,z^*\rangle$.

We have verified that \eqref{See:10} holds.
Since $(0,0)\in\gra A\cap\gra \partial f$, we obtain
$(\forall (d,d^*)\in\gra (A+\partial f))$ $\langle d,d^*\rangle\geq0$.
Thus, $F_{A+\partial f}(0,0)=0$.
Now define
\begin{align*}
j\colon \left[0,1\right]\rightarrow\RR \colon  t\rightarrow F_{A+\partial f}(tz,tz^*).
\end{align*}
Then $j$ is continuous on $\left[0,1\right]$ by \eqref{See:1a4} and \cite[Proposition~2.1.6]{Zalinescu}.
From \eqref{See:10}, we obtain
\begin{align}
F_{A+\partial f}(z,z^*)=\lim_{t\rightarrow1^{-}}
F_{A+\partial f}(t z,t z^*)\geq \lim_{t\rightarrow1^{-}}
\langle t z,t z^*\rangle=\langle z,z^*\rangle,
\end{align}
which contradicts \eqref{See:1a4}. Hence \begin{align}
F_{A+\partial f}(z,z^*)\geq \langle z,z^*\rangle.
\end{align}
Therefore,
 \eqref{e0:ourgoal} holds, and $A+\partial f$ is maximally monotone.
\end{proof}

\begin{remark}
In Theorem~\ref{PGV:1}, when $\inte\dom A\cap\dom\partial f\neq\varnothing$, we have $\dom A=X$ since $\dom A$ is a linear subspace.
 Therefore, we can verify
 the maximal monotonicity of $A+\partial f$ by  the Verona-Verona result
 (see  \cite[Corollary~2.9(a)]{VV}, \cite[Theorem~53.1]{Si2} or \cite[Corollary~3.7]{Yao3}).

\end{remark}

\begin{corollary}\label{domain:L1}
Let $A:X\To X^*$ be  a maximally monotone linear relation, and $f:X\rightarrow\RX$ be a
proper lower semicontinuous convex function with
 $\dom  A\cap\inte\dom\partial f\neq\varnothing$.
Then $A+\partial f$ is of type  $(FPV)$.

\end{corollary}
\begin{proof} By Theorem~\ref{PGV:1}, $A+\partial f$ is maximally monotone.
Let $C$ be a nonempty closed convex subset of $X$,
and suppose that $\dom (A+\partial f) \cap \inte C\neq \varnothing$.
Let $x_1\in \dom A \cap \inte\dom \partial f$ and $x_2\in \dom (A+\partial f) \cap \inte C$.
Thus, there exists   $\delta>0$ such that $x_1+\delta U_X\subseteq \dom f$
 and $x_2+\delta U_X\subseteq C$.
Then for small enough $\lambda\in\left]0,1\right[$, we have
 $x_2+\lambda (x_1-x_2)+\tfrac{1}{2}\delta U_X\subseteq C$.
Clearly, $x_2+\lambda(x_1-x_2)+\lambda\delta U_X\subseteq \dom f$.
Thus $x_2+\lambda(x_1-x_2)+\tfrac{\lambda\delta}{2}
U_X\subseteq \dom f\cap C=\dom (f+\iota_C) $.
By Fact~\ref{pheps:1},
  $x_2+\lambda(x_1-x_2)+\tfrac{\lambda\delta}{2} U_X\subseteq \dom \partial
(f+\iota_C)$.
Since $\dom A$ is convex, $x_2+\lambda(x_1-x_2)\in\dom A$
 and  $x_2+\lambda(x_1-x_2)\in\dom A\cap\inte\left[\dom \partial
(f+\iota_C)\right]$.
By Fact~\ref{f:F4} , $\partial f+N_C=\partial {(f+\iota_ C)}$.
Then, by Theorem~\ref{PGV:1} (applied to $A$ and $f+\iota_ C$),
$A+\partial f+N_C=A+\partial {(f+\iota_ C)}$ is maximally monotone.
By Fact~\ref{f:referee02a},   $A+\partial f$ is of type  $(FPV)$.
\end{proof}

\begin{corollary}[Simons]
\emph{(See \cite[Theorem~46.1]{Si2}.)}
Let $A:X\To X^*$ be a maximally monotone linear relation.
Then $A$ is of type (FPV).
\end{corollary}
\begin{proof}
Let $f=\iota_X$.  Then by Corollary~\ref{domain:L1},
 we have that  $A=A+\partial f$ is type of (FPV).
\end{proof}
\subsection{An example and comments}
\begin{example}
\label{ex:main}
Suppose that $X=L^1[0,1]$ with $\|\cdot\|_1$, let $$D=\menge{x\in X}{\text{$x$ is absolutely
continuous}, x(0)=0, x'\in X^*},$$
and set
\begin{equation*}
A\colon X\To X^*\colon
x\mapsto \begin{cases}
\{x'\}, &\text{if $x\in D$;}\\
\varnothing, &\text{otherwise.}
\end{cases}
\end{equation*}
Define $f:X\rightarrow\RX$ by
\begin{align}f(x)=
\begin{cases}\frac{1}{1-\|x\|_1^2},\,&\text{if}\, \|x\|<1;\\
+\infty,&\text{otherwise}.
\end{cases}
\end{align}
Clearly, $X$ is a nonreflexive Banach space.
By Phelps and Simons' \cite[Example~4.3]{PheSim},
$A$ is an at most single-valued maximally monotone linear relation
with proper dense domain, and $A$ is neither symmetric nor skew.
Since $g(t)=\tfrac{1}{1-t^2}$ is convex on the $\left[0,1\right[$
 (by $g''(t)=2(1-t^2)^{-2}+8t^2(1-t^2)^{-3}\geq 0, \forall t\in\left[0,1\right[$),
 $f$ is convex. Clearly, $f$ is proper lower semicontinuous, and by
  Fact~\ref{pheps:1}, we have
\begin{align}\dom f=U_X=\inte\dom f=\dom\partial f
=\inte\left[\dom \partial f\right].\label{LSD:039}\end{align}
Since $0\in\dom A\cap\inte\left[\dom \partial f\right]$,
 Theorem~\ref{PGV:1} implies that $A+\partial f$ is maximally monotone.
To the best of our knowledge, the maximal monotonicity of
$A+\partial f$ cannot be deduced from
any previously known result.
\end{example}

\begin{remark} To the best of our knowledge,
the results in \cite{VV,Voi1,ZalVoi, BWY9,Yao3}
 can not verify the maximal monotonicity in Example~\ref{ex:main}.

\begin{itemize}
\item Verona and Verona
(see \cite[Corollary~2.9(a)]{VV} or \cite[Theorem~53.1]{Si2} or \cite[Corollary~3.7]{Yao3})
 showed the following:
``Let $f: X\to \RX$ be proper, lower semicontinuous, and convex,
let $A: X\To X^*$ be maximally monotone, and suppose that
$\dom A=X$.
Then $\partial f +A$ is maximally monotone.''
The $\dom A$ in Example~\ref{ex:main} is proper dense, hence
$A+\partial f$ in Example~\ref{ex:main} cannot be deduced from the Verona -Verona result.

\item In \cite[Theorem~5.10($\eta$)]{Voi1}, Voisei showed that the sum theorem
 is true when $\dom A\cap\dom B$ is closed, $\overline{\dom A}$ is convex
and Rockafellar's constraint qualification holds.
 In Example~\ref{ex:main}, $\dom A\cap\dom \partial f$ is not closed by \eqref{LSD:039}.
Hence we cannot apply for \cite[Theorem~5.10($\eta$)]{Voi1}.

\item In \cite[Corollary~4]{ZalVoi}, Voisei and Z\u{a}linescu showed that
 the sum theorem is true when $^{ic}(\dom A)\neq\varnothing,
^{ic}(\dom B)\neq\varnothing$ and $0\in^{ic}\left[\dom A-\dom B\right]$.
 Since the $\dom A$ in Example~\ref{ex:main} is a proper dense linear subspace,
$^{ic}(\dom A)=\varnothing$. Thus  we cannot apply for \cite[Corollary~4]{ZalVoi}.
(Given a set $C\subseteq X$, we define $^{ic}C$ by
\begin{equation*}
^{ic}C=\begin{cases}^{i}C,\,&\text{if $\aff C$ is closed};\\
\varnothing,\,&\text{otherwise},
\end{cases}\end{equation*} where $^{i}C$ \cite{Zalinescu}
 is the \emph{intrinsic core} or \emph{relative algebraic interior} of $C$,  defined by
$^{i}C=\{a\in C\mid \forall x\in \aff(C-C),
\exists\delta>0, \forall\lambda\in\left[0,\delta\right]:
a+\lambda x\in C\}.$)

\item In \cite{BWY9},
it was shown that
the sum theorem is true when $A$ is a linear relation, $B$ is the subdifferential operator
of a proper lower semicontinuous sublinear function, and Rockafellar's
constraint qualification holds.
Clearly, $f$ in Example~\ref{ex:main} is not sublinear. Then we cannot apply for it.
Theorem~\ref{PGV:1} truly generalizes \cite{BWY9}.

\item In \cite[Corollary~3.11]{Yao3},
it was shown that
the sum theorem is true when $A$ is a linear relation, $B$ is a maximally monotone
operator satisfying  Rockafellar's
constraint qualification and $\dom A\cap\overline{\dom B}\subseteq\dom B$.
In Example~\ref{ex:main}, since $\dom A$ is a linear subspace, we can take $x_0\in\dom A$ with $
\|x_0\|=1$. Thus, by \eqref{LSD:039}, we have that \begin{align}
x_0\in\dom A\cap\overline{U_X}=\dom A\cap\overline{\dom \partial f}\quad\text{but}\quad  x_0\not\in U_X=\dom \partial f.
\end{align}
Thus $\dom A\cap\overline{\dom \partial f}\nsubseteqq \dom \partial f$ and thus we cannot apply
\cite[Corollary~3.11]{Yao3} either.
\end{itemize}
\end{remark}

\begin{problem}
Let $A:X\To X^*$ be a maximally monotone linear relation,
and let $f:X\To \RX$ be a proper lower semicontinuous convex function.
Assume that \begin{align}\bigcup_{\lambda>0} \lambda
\big(P_X(\dom F_A)-P_X(\dom F_{\partial f})\big)=X.\end{align}  Is
$A+\partial f$ necessarily maximally monotone $?$
\end{problem}


\begin{thebibliography}{99}

\bibitem{BB}
H.H.\ Bauschke and J.M.\ Borwein,
``Maximal monotonicity of dense type, local maximal monotonicity,
and monotonicity of the conjugate are all the same for continuous
linear operators'',
\emph{Pacific Journal of Mathematics},
vol.~189, pp.~1--20, 1999.

\bibitem{BBW}
H.H.\ Bauschke, J.M.\ Borwein, and X.\ Wang,
``Fitzpatrick functions and continuous linear
monotone operators'',
\emph{SIAM Journal on Optimization},
vol.~18, pp.~789--809, 2007.



\bibitem{BWY2}
H.H.\ Bauschke, X.\ Wang, and L.\ Yao,
``Autoconjugate representers for linear monotone
operators'',
\emph{Mathematical Programming (Series B)},
vol.~123, pp.~5--24, 2010.


\bibitem{BWY3}
H.H.\ Bauschke, X.\ Wang, and L.\ Yao,
``Monotone linear relations: maximality
and Fitzpatrick functions'',
\emph{Journal of Convex Analysis}, vol.~16, pp.~673--686, 2009.


\bibitem{BWY4}
H.H.\ Bauschke, X.\ Wang, and L.\ Yao,
``An answer to S.\ Simons' question
on the maximal monotonicity of the sum of a
maximal monotone linear operator and a normal cone operator'',
\emph{Set-Valued and Variational Analysis}, vol.~17, pp.~195--201, 2009.




\bibitem{BWY9}
H.H.\ Bauschke, X.\ Wang, and L.\ Yao,
``On the maximal monotonicity of the sum  of a maximal monotone
linear relation and the subdifferential operator
of a sublinear function'', to appear \emph{Proceedings of the Haifa Workshop on
 Optimization Theory and Related Topics.
Contemp. Math., Amer. Math. Soc., Providence, RI};\\
\texttt{http://arxiv.org/abs/1001.0257v1}, January  2010.

\bibitem{BWY7}
H.H.\ Bauschke, X.\ Wang, and L.\ Yao,
``Examples of discontinuous
maximal monotone linear operators
and the solution to a recent problem posed by B.F.~Svaiter'',
\emph{Journal of Mathematical Analysis and Applications},
 vol.~370, pp. 224-241,
 2010.


\bibitem{BWY8}
H.H.\ Bauschke, X.\ Wang, and L.\ Yao,
``On Borwein-Wiersma Decompositions of monotone linear
relations'',\emph{SIAM Journal on Optimization}, vol.~20, pp.~2636--2652,
 2010.




\bibitem{Bor1}
J.M.\ Borwein,
``Maximal monotonicity via convex analysis'',
\emph{Journal of Convex Analysis}, vol.~13, pp.~561--586, 2006.

\bibitem{Bor2}J.M.\ Borwein, ``Maximality of sums of two maximal monotone operators in general
Banach space'',
\emph{Proceedings of the
 American Mathematical Society}, vol.~135, pp.~3917--3924, 2007.

\bibitem{Bor3}J.M.\ Borwein, ``Fifty years of maximal monotonicity'',\\
http://www.carma.newcastle.edu.au/~jb616/fifty.pdf, January 2010.


\bibitem{BorVan}
J.M.\ Borwein and J.D.\ Vanderwerff,
\emph{Convex Functions},
Cambridge University Press, 2010.

\bibitem{BurIus}
R.S.\ Burachik and A.N.\ Iusem,
\emph{Set-Valued Mappings and Enlargements of Monotone Operators},
Springer-Verlag, 2008.


\bibitem{ButIus}
D.\ Butnariu and A.N.\ Iusem,
\emph{Totally Convex Functions for Fixed Points Computation
and Infinite Dimensional Optimization},
Kluwer Academic Publishers, 2000.

\bibitem{Cross}
R.\ Cross,
\emph{Multivalued Linear Operators},
Marcel Dekker, 1998.


\bibitem{Fitz88}
S.\ Fitzpatrick,
``Representing monotone operators by convex
functions'', in  \emph{Workshop/Miniconference on Functional Analysis
and Optimization (Canberra 1988)}, Proceedings of the Centre for
Mathematical Analysis, Australian National University, vol.~20,
Canberra, Australia, pp.~59--65, 1988.

\bibitem{MSV}M.M.\ Alves and B.F.\ Svaiter,
``A new proof for maximal monotonicity of subdifferential operators'',
\emph{Journal of Convex Analysis}, vol.~15, pp.~345--348, 2008.





\bibitem{ph}
R.R.\ Phelps,
\emph{Convex Functions, Monotone Operators and
Differentiability},
2nd Edition, Springer-Verlag, 1993.





\bibitem{PheSim}
R.R.\ Phelps and S.\ Simons, ``Unbounded linear monotone
operators on nonreflexive Banach spaces'',
\emph{Journal of Convex Analysis}, vol.~5, pp.~303--328, 1998.

\bibitem{Rock66}
R.T.\ Rockafellar,
``Extension of Fenchel's duality theorem for
convex functions'',
\emph{Duke Mathematical Journal}, vol.~33, pp.~81--89, 1966.

\bibitem{Rock70}
R.T.\ Rockafellar,
``On the maximality of sums of nonlinear monotone operators'',
\emph{Transactions of the American Mathematical Society},
vol.~149, pp.~75--88, 1970.
\bibitem{Rock702}
R.T.\ Rockafellar,
``On the maximality of sums of subdifferential mappings'',
\emph{Pacific Journal of Mathematics},
vol.~33, pp.~209--216, 1970.




\bibitem{RockWets}
R.T.\ Rockafellar and R.J-B Wets,
\emph{Variational Analysis}, 3nd Printing,
Springer-Verlag, 2009.


\bibitem{Rudin}
R.\ Rudin,
\emph{Functinal Analysis},
Second Edition, McGraw-Hill, 1991.



\bibitem{Si}
S.\  Simons,
\emph{Minimax and Monotonicity},
Springer-Verlag, 1998.

\bibitem{Si2}
S.\ Simons, \emph{From Hahn-Banach to Monotonicity},
Springer-Verlag, 2008.


\bibitem{Si3}
S.\ Simons,
``A Br\'{e}zis-Browder theorem for SSDB spaces'';\\
\texttt{http://arxiv.org/abs/1004.4251v3}, September 2010.




\bibitem{Svaiter}
B.F.\ Svaiter,
``Non-enlargeable operators and self-cancelling operators'',
\emph{Journal of Convex Analysis},
vol.~17, pp.~309--320, 2010.



\bibitem{VV1}
A.\ Verona and M.E.\ Verona,
``Regular maximal monotone operators'',
\emph{Set-Valued  Analysis},
vol.~6, pp.~303--312, 1998.


\bibitem{VV}
A.\ Verona and M.E.\ Verona,
``Regular maximal monotone operators and the sum theorem'',
\emph{Journal of Convex Analysis},
vol.~7, pp.~115--128, 2000.

\bibitem{Voisei06b}
M.D.\ Voisei,
``A maximality theorem for the sum of maximal monotone
operators in non-reflexive Banach spaces'',
\emph{Mathematical Sciences Research Journal},
vol.~10, pp.~36--41, 2006.


\bibitem{Voisei06}
M.D.\ Voisei,
``The sum theorem for linear maximal monotone operators'',
\emph{Mathematical Sciences Research Journal}, vol.~10, pp.~83--85, 2006.

\bibitem{Voi1}
M.D.\ Voisei,
``The sum and chain rules for maximal monotone operators",
\emph{Set-Valued and Variational Analysis}, vol.~16, pp.~461--476, 2008.






\bibitem{ZalVoi}
M.D.\ Voisei and C.\ Z\u{a}linescu,
``Maximal monotonicity criteria for the composition and the sum under weak interiority conditions'',
\emph{Mathematical Programming (Series B)},
vol.~123, pp.~265--283, 2010.

\bibitem{VZ}
M.D.\ Voisei and C.\ Z{\u{a}}linescu,
``Linear monotone subspaces of locally
convex spaces'',
\emph{Set-Valued and Variational Analysis},
vol.~18, pp.~29--55, 2010.

\bibitem{Yao}
L.\ Yao,
``The Br\'{e}zis-Browder Theorem revisited and properties of Fitzpatrick functions of order $n$'',
to appear \emph{Fixed Point Theory for Inverse Problems
in Science and Engineering (Banff 2009) , Springer-Verlag};\\
\texttt{http://arxiv.org/abs/0905.4056v1}, May 2009.



\bibitem{Yao3}
L.\ Yao,  ``The sum of a maximal monotone operator  of type (FPV) and a maximal monotone operator
with full domain is maximally monotone'', submitted;\\
\texttt{http://arxiv.org/abs/1005.2247v2}, August 2010.

\bibitem{Zalinescu}
{C.\ Z\u{a}linescu},
\emph{Convex Analysis in General Vector Spaces}, World Scientific
Publishing, 2002.

\bibitem{Zeidler}{E.\ Zeidler},
\emph{Nonlinear Functional  Analysis and its Application II/B: Nonlinear Monotone Operators},
Springer-Verlag, New York-Berlin-Heidelberg, 1990.
\end{thebibliography}
\end{document}